\documentclass[a4paper, final]{amsart}

\newcommand{\version}{1.9}

\newcommand{\mytitle}[2][]{%
  \gdef\mylongtitle{\MakeUppercase{#2}}
  \gdef\myshorttitle{\MakeUppercase{#1}}
  \gdef\mypdftitle{#1 v. \version}
}
\mytitle[Restrained structures]{Expansions of the reals which do not define
  the natural numbers}
\newcommand*{\mykeywords}{Definably complete, i-minimal, Hausdorff
  measure, restrained}

\title[\myshorttitle{} v. \version]{\mylongtitle\\
  \small{Version~\version}
}

\author[A. Fornasiero]{Antongiulio Fornasiero}
\address{Institut f\"ur Mathematische Logik\\
    Einsteinstr.~62, 48149 M\"unster, Germany}
\email{antongiulio.fornasiero@googlemail.com} 
\urladdr{http://www.dm.unipi.it/~fornasiero/}

\date{9 Apr 2011}

\usepackage{ifpdf}
\ifpdf
\usepackage[pdftex,
pdfborder={0 0 0}, plainpages=false,%
pdfpagelabels, pdfdisplaydoctitle=true,%
pdfstartview={XYZ null null null},%
colorlinks=false
]{hyperref}
\hypersetup{%
  pdfauthor={A. Fornasiero},
  pdftitle={mypdftitle{} v.\version},
  pdfkeywords={\mykeywords},
  pdfsubject={Definably complete structures such that every
    pseudo-enumerable set is nowhere dense}
}
\else
\usepackage[plainpages=false,linktocpage=true,pdfpagelabels, colorlinks=false
]{hyperref}
\fi

\usepackage{amsmath, amsthm, amssymb}
\usepackage{mathtools}
\usepackage{xspace}

\usepackage{comment}

\usepackage[alwaysadjust 
  ]{paralist}
\AtBeginDocument{\setdefaultleftmargin{\parindent}{\parindent}{}{}{}{}}
\setdefaultenum{\upshape(I)}{\upshape i)}{}{}

\usepackage[T1]{fontenc}

\usepackage[backrefs, 
alphabetic, msc-links
]{amsrefs}

\makeatletter
\@namedef{subjclassname@2010}{%
 \textup{2010} Mathematics Subject Classification}
\makeatother

\newcommand*{\intro}[1]{\textbf{#1}}
\newcommand*{\Pa}[1]{\bigl( #1 \bigr)}
\newcommand*{\set}[1]{\{\,#1\,\}}
\newcommand*{\iset}[1]{\{#1\}}

\newcommand{\N}{\mathbb{N}}
\newcommand{\Z}{\mathbb{Z}}
\newcommand{\R}{\mathbb{R}}

\newcommand{\Rbar}{\bar\R}
\newcommand*{\inter}[1]{\mathring{#1}}
\DeclareMathOperator{\interior}{int}
\newcommand*{\cl}[1]{\overline{#1}}
\DeclareMathOperator{\cll}{cl}

\newcommand{\rest}{\upharpoonright}

\newcommand*{\pair}[1]{\langle #1 \rangle}

\newcommand{\av}{{\bar a}}
\newcommand{\bv}{{\bar b}}
\newcommand{\cv}{{\bar c}}
\newcommand{\dv}{{\bar d}}
\newcommand{\ev}{{\bar e}}
\newcommand{\x}{{\bar x}}
\newcommand{\y}{{\bar y}}
\newcommand{\z}{{\bar z}}

\def\Ind#1#2{#1\setbox0=\hbox{$#1x$}\kern\wd0\hbox to
  0pt{\hss$#1\mid$\hss}\lower.9\ht0\hbox to 0pt{\hss$#1\smile$\hss}\kern\wd0}

\newcommand{\et}{\ \&\ }
\newcommand{\vel}{\ \vee\ }

\newcommand{\Cp}{\mathcal C^p}
\newcommand{\Cone}{\mathcal C^1}

\def\hyph{\nobreakdash-\hspace{0pt}\relax}

\newcommand{\Wlog}{W.l.o.g\mbox{.}\xspace}
\newcommand{\wloG}{w.l.o.g\mbox{.}\xspace}

\newcommand{\ie}{i.e\mbox{.}\xspace}

\newcommand{\cf}{cf\mbox{.}\xspace}

\newcommand{\Tfae}{T.f.a.e\mbox{.}\xspace}

\newtheorem{lemma}{Lemma}[section]
\newtheorem{theorem}[lemma]{Theorem}
\newtheorem{corollary}[lemma]{Corollary}
\newtheorem{conjecture}[lemma]{Conjecture}

\newtheorem{open problem}[lemma]{Open problem}

\newtheorem*{proviso*}{Proviso}
\newtheorem*{fact*}{Fact}
\newtheorem{fact}[lemma]{Fact}

\theoremstyle{remark}

\newtheorem*{claim*}{Claim}
\newenvironment{prclaim}{\begin{inparaenum}[\textit{Claim} 1.\hspace{.7ex}]\item}
  {\end{inparaenum}\par\indent}
\newtheorem{exercise}[lemma]{Exercise}

\theoremstyle{definition}
\newtheorem{definition}[lemma]{Definition}
\newtheorem{notation}[lemma]{Notation}
\newtheorem{remark}[lemma]{Remark}
\newtheorem{final remark}[lemma]{Final remark}

\newtheorem{question}[lemma]{Question}

\newenvironment{sentence}[1][]{%
  \begin{list}{}{%
    \setlength\topsep{0.5ex}%
    \setlength\leftmargin{\parindent}%
  }%
  \item[#1]
 }
 {\end{list}}

\newcommand{\Ds}{\mathcal D_\Sigma}
\newcommand{\Fs}{\mathcal F_\sigma}
\newcommand{\Gdelta}{\mathcal G_\delta}
\newcommand{\K}{\mathbb K}
\newcommand{\Rc}{\mathcal R}
\newcommand{\iminimal}{i\hyph minimal\xspace}
\newcommand{\dcompact}{d\hyph compact\xspace}
\newcommand{\pN}{pseudo\hyph$\N$\xspace}
\newcommand{\psenum}{pseudo\hyph enumerable\xspace}
\newcommand{\dminimal}{d\hyph minimal\xspace}
\newcommand{\dminimality}{d\hyph minimality\xspace}
\newcommand{\ominimal}{o\hyph minimal\xspace}
\newcommand{\Zdefinable}{$\emptyset$\hyph definable\xspace}
\newcommand{\Rcal}{\mathcal R}

\newcommand{\Disc}{\mathfrak D}
\newcommand{\Bad}{\mathfrak B}
\newcommand{\cPic}{\sideset{^c}{^c}{\operatorname{\Pi}}}
\newcommand{\ao}{a.o\mbox{.}\xspace}
\newcommand{\Sdiff}{\mathbin{\Delta}}
\DeclareMathOperator{\lexmin}{lex\, min}
\DeclareMathOperator{\Fcl}{Fcl}
\DeclareMathOperator{\rk}{rk}
\newcounter{saveenum}
\newcommand*{\romref}[1]{{\rm{\ref{#1}}}}
\newcommand*{\enumref}[1]{\rm{(\ref{#1})}}
\newcommand{\dimH}{\dim_{\mathcal H}}
\newcommand{\eps}{\varepsilon}

\newcommand{\tame}{restrained\xspace}

\begin{document}

\begin{abstract}
We study first-order expansions of the reals which do not define the set of
natural numbers.
We also show that several stronger notions of tameness are
equivalent to each others.
\end{abstract}

\maketitle


\makeatletter
\renewcommand\@makefnmark%
   {\normalfont(\@textsuperscript{\normalfont\@thefnmark})}
\renewcommand\@makefntext[1]%
   {\noindent\makebox[1.8em][r]{\@makefnmark\ }#1}
\makeatother

\section{Introduction}

O-minimal structures were introduced in the '80s as a framework for 
``tame topology'' (see \cites{vdd,DM96}).
Many expansions of $\Rbar$ have been shown to be \ominimal:
some important examples are $\pair{\R_{an}, \exp}$, the expansion of $\Rbar$ by
restricted analytic functions and the exponential function
(see~\cite{DMM94}), and the Pfaffian closure of $\pair{\R_{an},\exp}$
(see~\cite{speissegger99}).

Given a first-order structure $\K$, unless otherwise specified,
by ``definable'' we will always mean ``definable with parameters from $\K$''.
$\Rc$~will be an expansion of the field of real numbers~$\Rbar$.

If $\Rc$ is \ominimal, then every definable set has finitely many
connected components; most importantly, the dimension 
(see Definition~\ref{def:dim}) is well-behaved 
(that is, it satisfies the axioms (Dim 1--4) in \cite{dries89}:
see also Conjecture~\ref{conj:dim}),
and moreover satisfies, for every nonempty definable set~$X$, 
\begin{enumerate}[({Dim} 1)]
\setcounter{enumi}{4}
\item\label{en:dim-finite} $\dim(X) = 0$ iff $X$ is finite;
\item\label{en:dim-frontier} $\dim(\partial X) < \dim (X)$.
\end{enumerate}

In his article \cite{miller05}, C.~Miller studied several classes of structures
expanding $\Rbar$ which still present a ``tame'' behaviour, 
without being \ominimal.
In this article, we will focus on two such classes: \intro{\iminimal{}}
structures and \intro{\tame{}} structures.

One of the first examples of \tame non \ominimal structures
was given by \dminimal expansions of~$\Rbar$ (see
\cites{dries85, FM05, MT06}), that is structures such that every definable
subset $X$ of $\R$ with empty interior is the union of finitely many discrete
sets, and the number of discrete sets does not depend on the parameters of
definition of~$X$: for instance, $\pair{\Rbar, 2^\Z}$ is a \dminimal structure
which is not \ominimal.
D-minimal structures satisfy many of the properties of \ominimal structures;
most importantly, the dimension is well-behaved
(see \cites{miller05, tame:3} and Theorem~\ref{mainthm:2}),
but the additional properties (Dim~\ref{en:dim-finite}) and
(Dim~\ref{en:dim-frontier}) do not  hold. 

Remember that a subset $X$ of a topological space $Y$ is nowhere dense 
(in~$Y$) if the closure of $X$ has empty interior,
and that a subset of $\R^n$ is null if it has Lebesgue measure~$0$.

\begin{theorem}\label{mainthm:3}
\Tfae:
\begin{enumerate}
\item\label{en:main3-imin} 
$\mathcal R$ is \intro{\iminimal{}}
(that is, every definable subset of $\R$ has interior or is nowhere dense);
\item\label{en:main3-null}
every definable subset of $\R$ has interior or is null;
\item\label{en:main3-Haus}
every definable subset of $\R$ has interior or has Hausdorff dimension~$0$.
\end{enumerate}
\end{theorem}
The above  theorem shows that several ``natural'' strengthenings of
\dminimality are equivalent to each other, and
solves an open problem in \cite{miller05}*{\S3.1}
(and, implicitly, in \cite{FKMS} and \cite{FM01}, which also give some 
examples and general methods of constructing \iminimal expansions of~$\Rbar$).
See also \cites{tame:3, miller05} for some more properties of \iminimal
structures. 
We can give a higher-dimensional analogue of Theorem~\ref{mainthm:3}, but
first we need to define the dimension.
\begin{definition}\label{def:dim}
Let $\K$ be an expansion of an ordered field, and $X \subseteq \K^n$.
The \intro{dimension} of $X$ is $\dim(X)$, 
the maximum $e$ such that there exists a coordinate space $L$
of linear dimension~$e$, such that $\Pi^n_L(X)$ has nonempty interior (inside
$L$), where $\Pi^n_L$ is the orthogonal projection onto $L$.
\end{definition}

We denote by $\dimH$ the Hausdorff dimension.

\begin{theorem}\label{mainthm:imin-dim-n}
Let $\Rc$ be \iminimal and $C \subseteq \R^n$ be definable and nonempty.
Then, $C$~is Lebesgue measurable, almost open, and $\dimH(C) = \dim(C)$.
\end{theorem}

I am grateful to C. Miller for allowing me to present the following
result and its proof (see \S\ref{sec:real} for the proof).
\begin{theorem}[C. Miller]\label{mainthm:sparse}
Let $\Rcal$ be an \ominimal expansion of~$\Rbar$.
Let $E \subset \R^n$ be closed, such that $\dim(E) = 0$. 
Denote  by $\pair{\Rcal, E}$ the expansion $\pair{\Rcal, (Y)}$ of~$\Rcal$,
where $Y$ ranges among all subsets of finite Cartesian powers of~$E$.
Then, either $\pair{\Rcal, E}$ defines~$\N$, 
or $\pair{\Rcal, E}^\#$ is \iminimal.
\end{theorem}

An alternative way of extending o-minimality is given by structures with
\ominimal open core.
\begin{definition}
The \intro{open core} of $\K$ is the reduct of $\K$ generated by all open
definable subset of $\K^n$ (for all $n \in \N$).
\end{definition}
The main examples of structures with \ominimal open core
are given by dense elementary pairs of \ominimal structures
(see~\cite{vdd-dense}, but see also \cites{DMS, HG11} for other examples); 
their properties have been studied in \cites{DMS, tame:1}.
The dimension will no longer be ``well-behaved'' on such structures (for
instance, the union of two definable sets of dimension $0$ can have
dimension~$1$), but its restriction to the class of sets definable in the open
core will be well-behaved (but see again Conjecture~\ref{conj:dim}).

All the above examples are particular cases of \tame structures.
\begin{definition}\label{def:tame-R}
We call $\Rc$ \intro{\tame{}} if, for every definable discrete set 
$D \subset \R^n$ and
every definable function $g: D \to \R$,  $g(D)$ is nowhere dense (in~$\R$).
\end{definition}
See \cites{BZ08,Hieronymi11} for more examples of \tame expansions of~$\Rbar$.

On the one hand, non-\tame expansions of $\Rbar$ are ``wild'' (from a
model-theoretic point of view; \cf \S\ref{sec:pathology}\eqref{en:RN}).
If $\Rc$ defines the set $\N$, then $\Rc$ is not \tame
(since, if we take $D \coloneqq \N \times \N_{>0}$ and $f(x,y) \coloneqq
x/y$, we have $f(D)$ dense in $\R_{\geq 0}$).
The converse is the main reason of interest in restrained structures: 
\begin{fact}[\cite{hieronymi10}*{Theorem~1.1}]\label{fact:hier-tame}
Either $\Rc$ is \tame, or $\Rc$ defines $\N$.
\end{fact}
\begin{corollary}
If the theory of $\Rc$ is decidable, then $\Rc$ is \tame.
\end{corollary}
Notice that the converse of the above corollary is false: there exist \ominimal
expansions of $\Rbar$ which are not decidable (for instance, it suffices to
add to $\Rbar$ a constant for a suitable real number).

On the other hand, what we call \tame structures do present a tame behaviour:
more precisely, a certain class of sets definable in \tame structures behaves in
a controlled way.

\begin{definition}
Let $X \subseteq \R^n$ be definable (in~$\Rc$).
We say that $X$ is a $\Ds$ set if there exists a definable closed set $Y
\subseteq \R^{n+1}$, such that
$\Pi^{n+1}_n(Y) = X$,
where $\Pi^{n+1}_n$ is the projection onto the first $n$ coordinates.
\end{definition}
Notice that if $X$ is a Boolean combination of definable closed sets, then
$X$ is~$\Ds$.
Moreover, if $\Rc$ is \dminimal, then every definable set is a $\Ds$ set.

\begin{theorem}
\label{mainthm:1}
Let $\Rc$ be \tame.
Let $D \subseteq \R^n$ be $\Ds$ and nonempty.
\Tfae:
\begin{enumerate}
\item $D$ has empty interior;
\item $D$ is nowhere dense;
\item $D$ is null.
\end{enumerate}
Moreover, 
$\dimH(D) = \dim(D)$.
\end{theorem}

While our main focus for this article is in expansions of $\Rbar$,
the proofs are easier if we work in sufficiently saturated structures
(see \S\ref{sec:dimension}); moreover, many of our results extend to
definably complete structures.
\begin{definition}
Let $\K$ be an expansion of an ordered field.
$\K$ is \intro{definably complete} (DC) if every definable subset of $\K$ has
a least upper bound in $\K \cup \iset{\pm \infty}$.
\end{definition}
Definably complete structures were introduced in \cite{miller},
and have been studied (among other places) in
\cites{DMS,frat,AF11,FS:1,tame:1,tame:2, tame:3}; see also \S\ref{sec:prelim}
for some properties of DC structures.

$\K$ will always be a DC structure (expanding an ordered field).
We can generalize Definition~\ref{def:tame-R} to arbitrary DC structures.
\begin{definition}
$\K$ is \intro{\tame} if (it is a DC expansion of an ordered field and), 
for every definable discrete set $D \subset \K^n$ and
every definable function $g: D \to \K$, $g(D)$ is nowhere dense (in~$\K$).
\end{definition}


The main result that allows us to prove the above theorems
is the fact that the dimension function is well-behaved on
$\Ds$ set (see Definition~\ref{def:Ds}), provided that $\K$ is \tame.
\begin{theorem}\label{mainthm:2}
Let $\K$ be \tame.
Let $A, A' \subseteq \K^n$ be a $\Ds$ set.
\begin{enumerate}
\item
$\dim(\cl A) = \dim(A)$ (where $\cl A$ is the topological closure of~$A$).
\item\label{en:main-dim-function}
 Let $f: A \to \K^n$ be definable and continuous.
Then, $\dim(f(A)) \leq \dim(A)$.
\item\label{en:main-dim-union} 
$\dim(A \cup A') = \max\Pa{\dim(A), \dim(A')}$.
\item
Let $B \subseteq \K^{n+m}$ be a $\Ds$ set.
Let $A \coloneqq \Pi^{n+m}_n(B)$.
Assume that $\dim(A) = p$ and $\dim(B_\av) = q$ for every $\av \in A$
(where $B_\av$ is the fiber of $B$ over $\av$).
Then, $\dim(B) = p + q$.
\end{enumerate}
\end{theorem}
The proof of the above theorem is easier if we work inside  $\omega$-saturated
structures (hence we formulated it for $\K$ instead of~$\Rc$).

Some additional results on \tame structure, while not used in the proof of the
previous theorems, may nevertheless be interesting on their own.
In particular, we have the following further example of tame behaviour:
\begin{theorem}\label{mainthm:Cp}
Let $\K$ be \tame.
Let $f: \K^n \to \K^m$ be definable.
\begin{enumerate}
\item\label{en:Cp-n} 
If $f$ is continuous and $p \in \N$, then
there exists a definable nowhere dense set $D \subset \K^n$, such that
$f$ is $\Cp$ outside $D$.
\item\label{en:Sard}
If $f$ is $\Cone$, then $\Sigma_f$, the set of singular values of~$f$, 
is nowhere dense.
\end{enumerate}
\end{theorem}

While for expansions of $\Rbar$ we have Fact~\ref{fact:hier-tame}, for general
DC structures we have only a conjecture.
\begin{conjecture}
Either $\K$ is \tame, or it defines a discrete subring (containing~$1$).
\end{conjecture}

\subsection*{Acknowledgements}
Thanks to Philipp L\"ucke, Ben Miller, Chris Miller, and Tamara Servi.

\section{Preliminaries}\label{sec:prelim}
\begin{proviso*}
$\K$ will always be a DC structure expanding a field.
\end{proviso*}
We will freely use the results on definably complete structures contained in
\cite{miller} and \cite{tame:1}*{\S 2--5}.

\begin{notation}
We denote by $\Pi^n_m: \K^n \to \K^m$ the projection onto the first $m$
coordinates.
Given a linear space $L \subseteq \K^n$, we denote by
$\Pi^n_L: \K^n \to L$ the orthogonal projection onto~$L$.
Given $B \subseteq \K^{n+m}$ and $\av \in \K^n$,
we denote by $B_\av \coloneqq \set{\cv \in \K^m: \pair{\av, \cv} \in B}$ 
the corresponding fiber of~$B$.
\end{notation}

In the proofs, we will always assume that there exists a \pN set,
that is a definable, discrete, closed, and unbounded subset of $\K$,
which we will denote by~$N$.
If such a set does not exists, then $\K$ has locally \ominimal open core
\cite{tame:1} (see also \cite{DMS}), 
and one can easily verify each result in that case.

\begin{definition}[\cite{FS:1}]
Let $X \subseteq \K^n$ be definable.
We say that $X$ is \intro{definably meager} (in $\K^n$) if there is no family 
$\Xi \coloneqq \Pa{X_t: t \in \K}$ of closed subsets of $\K^n$,
such that $\Xi$ is definable, increasing, each $X_t$ is nowhere dense,
and $X \subseteq \bigcup_{t \in \K} X_t$.
\end{definition}

In DC structures we have an analogue of Baire Category Theorem.
\begin{fact}
For every $n \in \N$, $\K^n$ is not definably meager (in itself).
\end{fact}
\begin{proof}
See \cite{hier:BCT} and \cite{FS:1}*{Proposition~2.14}.
Notice that for this article we are interested only in the case when $\K$ is
\tame, where the proof of the above fact is much easier
(see~\cite{tame:2}*{\S6}).
\end{proof}

We have a generalization of $\Ds$ sets to definably complete structures.
\begin{definition}\label{def:Ds}
Let $X \subseteq \K^n$ be definable.
We say that $X$ is a $\Ds$ set if there exists a definable closed set $Y
\subseteq \K^{n+1}$, such that
$\Pi^{n+1}_n(Y) = X$.%
\footnote{In \cite{FS:1} we called $\Ds$ sets ``definably $\Fs$ sets''.
However, since the main focus here is on expansions of $\R$, the previous
nomenclature might be confusing (since a set which is definable and $\Fs$ is
not necessarily $\Ds$); hence, we adopt instead the nomenclature
from~\cite{MS99}, which, while being less suggestive, is also less prone to misunderstanding.}
\end{definition}

\begin{fact}[\cite{FS:1}]\label{fact:meager}
Let $X \subseteq \K^n$ be definable.
Then, $X$ is definably meager iff there exists $Y \subset \K^n$
such that $Y$ is $\Ds$, with empty interior, and $X \subseteq Y$.
\end{fact}

We can already prove Theorem~\ref{mainthm:2}\eqref{en:main-dim-union}, 
which holds for arbitrary DC structures.
\begin{lemma}
Let $A, A' \subseteq \K^n$ be $\Ds$.
Then, $\dim(A \cup A') = \max\Pa{\dim(A), \dim(A')}$.
\end{lemma}
\begin{proof}
Let $B \coloneqq A \cup A'$.
It is clear that $\dim(B) \geq \max\Pa{\dim(A), \dim(A')}$.
Assume, for a contradiction, that $\dim(B) > \max\Pa{\dim(A), \dim(A')}$.
\Wlog, $\Pi^n_d(B)$ has nonempty interior.
But $\Pi^n_d(B) = \Pi^n_d(A) \cup \Pi^n_d(B)$; by assumption and 
Fact~\ref{fact:meager}, $\Pi^n_d(A)$ and $\Pi^n_d(A')$ are both definably
meager;
thus, $\Pi^n_d(B)$ is also definably meager, absurd.
\end{proof}

\begin{fact}[\cite{FS:1}*{Lemma~3.10}]\label{fact:Fs-function}
Let $f: \K^n \to \K^m$ be a definable function, such that its graph
$\Gamma(f)$ is a $\Ds$ set.
Then, $\Disc(f)$, the set of discontinuity points of~$f$, is definably meager.
\end{fact}

\begin{fact}[\cite{tame:3}*{\S2}]\label{fact:first-class}
Let $f: \K^n \to \K$ be 
such that $f$ is the point-wise limit of a definable family of continuous
functions. 
Then, $\Disc(f)$ is definably meager.
\end{fact}

\begin{definition}
Let $X \subseteq \K^n$.
We say that $X$ is \dcompact if $X$ is definable, closed, and bounded.
\end{definition}

D-compact sets are the definable analogue of compact subsets of $\R^n$
(see \cite{miller}).
their main property is given by the following fact.

\begin{fact}[\cite{miller}]
Let $X$ be \dcompact and $f: X \to \K^m$ be definable and continuous.
Then, $f(X)$ is also \dcompact.
\end{fact}

We use the following notation: given $m \geq n$,
$A \subseteq \K^{m}$, and a map $\Pi:\K^{m} \to \K^n$, define
$\cPic(A) \coloneqq \K^n \setminus \Pa{\Pi(\K^{m} \setminus A)}$.
The following remark help us to easily show that certain definable sets 
are~$\Ds$, by looking at the formulae defining them.

\begin{remark}\label{rem:formula}
1) Let $\phi(\x)$ be a formula of the form
\[
(Q_1 y_1) \dotsc (Q_n y_m) \psi(x, y_1, \dotsc, y_m), 
\]
where $\psi(\x, \y)$ is some formula, and each $Q_i$ is a
quantifier, either $\exists$ or $\forall$.  
Let $A \subseteq \K^n$ and $B \subseteq \K^{n + m}$ 
be the sets defined by $\phi$ and $\psi$ respectively.
Then, $A = \Pi_1^*( \dotsc \Pi_m^*(B) \dots )$, 
where each $\Pi_i^*$ is either the
orthogonal projection $\Pi: \K^{n+i} \to \K^{n+i-1}$ onto the first $(n+i-1)$
coordinates, or $\cPic$, according to whether $Q_i$ is the quantifier
$\exists$, or the quantifier $\forall$.
\par\noindent
2)
Let $B \subseteq \K^{n+m}$ be definable:
\begin{enumerate}[i)]
\item
If $B$ is \dcompact, then $\Pi(B)$ is \dcompact;
\item 
If $B$ is closed, or more generally a $\Ds$ set, then $\Pi(B)$
is a $\Ds$ set; 
\item
if $B$ is open, then $\Pi(B)$ is also open;
\item if $B$ is closed, then $\cPic(B)$ is also closed;
\item
if $B$ is open, or more
generally a definably $\Gdelta$ set (\ie, the complement of a $\Ds$ set), 
then $\cPic(B)$ is a definably $\Gdelta$ set.
\end{enumerate}
\end{remark}

It is not clear if the union of a definable family of $\Ds$ (resp., definably
meager) sets, indexed by $N$, is $\Ds$ (resp., definably meager).
To prove it we need some additional ``uniformity'' assumptions, which will be
always satisfied in all our applications.
\begin{definition}
Let $\Xi \coloneqq \Pa{X_i: i \in I}$ be a definable family of subsets of
$\K^n$.
We say that $\Xi$ is a \intro{strongly uniform} family of $\Ds$ sets (resp., of
definably meager sets) if there exists a definable family
$\Pa{Y_i: i \in I}$ of closed subsets of $\K^{n+1}$, such that, 
for every $i \in I$,
$X_i = \Pi^{n+1}_n(Y_i)$ (resp., $X_i \subseteq \Pi^{n+1}_n(Y_i)$, and 
$\Pi^{n+1}_n(Y_i)$ has empty interior).
\end{definition}

\begin{remark}
Let $\Xi \coloneqq \Pa{X_n: n \in N}$ be a definable family of subsets of
$\K^n$.
If $\Xi$ is a strongly uniform family of $\Ds$ (resp., definably meager) sets,
then $\bigcup_{n \in N} X_n$ is also a $\Ds$ (resp., definably meager) set.
\end{remark}

\begin{lemma}\label{lem:dim-fiber-0}
Let $A \subseteq \K^n$ be a $\Ds$ set, and $d \coloneqq \dim(A)$.
Then, the set $B \coloneqq \set{\x \in \K^d: \dim(A_{\x}) > 0}$ is a definably
meager $\Ds$ set.
\end{lemma}
\begin{proof}[Sketch of proof]
The fact that $B$ is definably meager is a consequence of the (definable
version of) Kuratowski-Ulam theorem: see \cite{MS99}*{1.5(4)} and
\cite{FS:1}*{\S4} for details. 
If $A$ is closed, then, by Remark~\ref{rem:formula},
$B$ is a $\Ds$ set.
If $A$ is definably $\Ds$, then $A = \bigcup_{i \in N} A(i)$,
for some $\Pa{A(i): i \in N}$ definable family of \dcompact sets.
Then, $B = \bigcup_{i \in N} B(i)$, where $B(i) \coloneqq \set{\x \in \K^d:
  \dim(A(i)_{\x}) > 0}$.
By the previous case, each $A(i)$ is a $\Ds$ set.
Moreover, as it is easy to check, the family $\Pa{B(i): i \in N}$ is a
strongly uniform family of $\Ds$ sets; thus, $B$ is $\Ds$.
\end{proof}

\begin{lemma}\label{lem:dim-fiber-p}
Let $A \subseteq \K^{n+m}$ be a $\Ds$ set.
Let $B^{\geq p} \coloneqq \set{\x \in \K^d: \dim(A_{\x}) \geq p}$.
Then $B^{\geq p}$ is a $\Ds$ set.
\end{lemma}
\begin{proof}
Assume that $A$ is \dcompact.
Then, $B^{\geq p}$ is the union of finitely many sets, which, after a
permutation of coordinates, are of the form
\[
B(i) \coloneqq \set{\bv \in \K^n: \exists \cv \in \K^p\ \exists r > 0\ 
\forall \y \in \K^p\ \Pa{\y \notin B(\cv; r) \vel \exists \z \in \K^{n-p} 
\pair{\bv, \y, \z} \in A}}.
\]
By Remark~\ref{rem:formula}, each $B(i)$ is a $\Ds$ set,
and thus $B^{\geq p}$ is also a $\Ds$ set.

If instead $A$ is a $\Ds$ set, proceed as in
the proof of Lemma~\ref{lem:dim-fiber-0}.
\end{proof}

\begin{lemma}\label{lem:compact-uniformization}
Let $A \subseteq \K^n \times \K^p$ be a \dcompact set.
Let $B \coloneqq \Pi^{n+p}_n(A)$.
Define
$f: B \to \K^p$, $f(x) \coloneqq \lexmin(A_x)$.
Then, $\Gamma(f)$ is a definably $\Gdelta$ set, and $\Disc(f)$ is definably meager.
%
\end{lemma}
\begin{proof}
Let $\Pi \coloneqq \Pi^{n+p}_n$.
For simplicity, we will treat only the case $p = 1$.
Notice that
\[
\Gamma(f) = A \cap \cPic \set{\pair{\x, y, y'}: 
y > y' \rightarrow \pair{\x, y'} \notin A}.
\]
The fact that $\Gamma(f)$ is definably $\Gdelta$ follows from
Remark~\ref{rem:formula}.
$\Disc(f)$ is definably meager by \cite{DMS}*{Lemma~2.8(1)}.
\end{proof}

Notice that there are some compact subsets of $\R^2$, such that, for the
corresponding function $f$ as in the above lemma, 
$\Gamma(f)$ is a $\Gdelta$, but not an $\mathcal F_\sigma$ set.

\begin{definition}\cite{tame:2}*{\S4}
A definable set $X$ is \intro{at most \psenum{}} if there exists a definable
discrete set $D$ and a definable surjective map $g: D \to X$.
\end{definition}
Hence, $\K$ is \tame iff every at most pseudo-enumerable subset of $\K$ is
nowhere dense.

\begin{fact}[\cite{tame:2}*{\S5}]\label{fact:ps-enum}
\begin{enumerate}\item 
Let $A \subseteq \K^n$ be at most \psenum.  
Then, there exist $M \subset \K$
definable, closed and discrete, and $g: M \to \K^n$ definable, such that 
$A = g(M)$.
\item
Let $A$, $A'$ be at most \psenum subsets of $\K^n$.
Then, $A \cup A'$ and $A \times A'$ are also at most \psenum.
\item
Let $A$ be at most \psenum and $B \subseteq A$ be definable.
Then, $B$ is at most \psenum.
\end{enumerate}
\end{fact}

\begin{lemma}\label{lem:nd-countable}
Let $X \subseteq \K$ be nowhere dense.
Then, there exists a set $Y \subset \K$ discrete, definable, and such that
$\cl X = \cl Y$.
Moreover, the choice of $Y$ can be made in a strongly uniform way:
that is, if $Z \subset \K^{n+1}$ is definable, and for every $t \in \K^n$,
$Z_t$ is nowhere dense, then there exists $W \subset \K^{n+1}$ definable,
such that, for every $t \in \K^n$, $W_t$ is discrete, and $Z_t \subseteq \cll(W_t)$.
\end{lemma}
\begin{proof}
\Wlog, we can assume that $X$ is closed.
Take $Y$ to be the set of endpoints of $\K \setminus X$ (see
\cite{tame:2}*{\S2}). 
\end{proof}

\subsection{Proof of Theorem~\ref{mainthm:3}}
The proof of Theorem~\ref{mainthm:3} is surprisingly straightforward.
($\romref{en:main3-Haus} \Rightarrow \romref{en:main3-null}$) is clear
and ($\romref{en:main3-null} \Rightarrow \romref{en:main3-imin}$) is easy (see
\cite{miller05}*{3.1}).

($\romref{en:main3-imin} \Rightarrow \romref{en:main3-Haus}$): let $A \subset
\R$ be definable and with empty interior; we have to show that $A$ has
Hausdorff dimension~$0$.  Since we assumed that $\Rcal$ is \iminimal, we can
substitute $A$ with its closure, and thus, \wloG, $A$~is closed.  By
Lemma~\ref{lem:nd-countable}, there is a countable definable set $Y \subseteq
A$ such that $A = \cl Y$.  Assume, for a contradiction, that $\dimH(A) > 0$.
By \cite{EM01}*{Lemma~1}, there exists a linear function $T: \R^n \to \R$, such
that $T(A^n)$ has interior (in~$\R$).
Since $Y$ is dense in $A$ and $T$ is continuous, $Z \coloneqq T(Y^n)$ is
somewhere dense; since $\Rcal$ is \iminimal, $Z$ has nonempty interior,
contradicting the fact that $Z$ is countable.
\hfill\qedsymbol

\section[Meager and nowhere dense sets]{Meager and nowhere dense sets in \tame structures}

\begin{lemma}\label{lem:equivalence}
$\K$ is \tame iff, for every definably meager set $X \subset \K$,
we have that $X$~is nowhere dense.
\end{lemma}
\begin{proof}
%
For the ``only if'' direction,
let $X \subseteq \bigcup_{i \in N} Y_i$, 
with $(Y_i: i \in \N)$ definable family of nowhere dense sets.
By Lemma~\ref{lem:nd-countable}, there exists a definable family of discrete
sets $(Z_i: i \in N)$, such that, for every $i \in N$, 
$Y_i \subseteq \cll(Z_i)$.
Let $W \coloneqq \bigcup_{i \in N} Z_i$.
By \cite{tame:2}*{\S 5}, $W$~is at most pseudo-enumerable.
Since $X \subseteq \cl W$, we have that $X$ is nowhere dense.

For the ``if'' direction,
let $X \subset \K$ be at most pseudo-enumerable.
Then, $X$ is definably meager;
thus, by assumption, 
$X$~is nowhere dense, proving that $\K$ is \tame.
%
\end{proof}

\begin{proviso*}
From now on, we will assume that $\K$ is \tame (besides being a definably
complete expansion of an ordered field).
\end{proviso*}

The following lemma includes a higher-dimensional analogue of
Lemma~\ref{lem:equivalence}.
\begin{lemma}\label{res:meager-n}
\begin{enumerate}
\item\label{en:equiv-n-meager} 
Let $X \subseteq \K^n$ be definably meager.  Then, $X$ is nowhere dense.
\item\label{en:equiv-n-nd}
Let $X \subseteq \K^n$ be a $\Ds$ set.
Then, either $X$ has interior or it is nowhere dense.
\item\label{en:equiv-n-bad}
Let $A \subseteq \K^{n+p}$ be a $\Ds$ set.
The set
\[
\Bad_n(A) \coloneqq \set{x \in \K^n: \cll(A)_x \neq \cll(A_x)}
\]
is nowhere dense.
\end{enumerate}
\end{lemma}
\begin{proof}[Sketch of proof]
\eqref{en:equiv-n-meager} is equivalent to \enumref{en:equiv-n-nd}, 
since, if $X$ is a $\Ds$ set,
then $X \setminus \inter X$ is a definably meager set and a $\Ds$ set.

The proofs of \enumref{en:equiv-n-nd} and \enumref{en:equiv-n-bad} 
proceed together by induction on $n$:
the case $\enumref{en:equiv-n-nd}_1$ is Lemma~\ref{lem:equivalence},
while the proofs of $\enumref{en:equiv-n-bad}_1$ 
and the inductive step are, with minor
modifications, the same as in \cite{MS99}*{1.6}.
\end{proof}

\begin{definition}[\cite{FS:1}*{\S5}]
Let $X \subseteq \K^n$ be a definable.
We say that $X$ is \intro{definably almost open}, or \ao for brevity, if
there exist a definable open set
$U$ and a definably meager set $F$, such that $X = U \Sdiff F$.
\end{definition}

\begin{fact}[\cite{FS:1}*{\S5}]
The family of \ao subset of $\K^n$ is a Boolean algebra containing all
$\Ds$ sets.
\end{fact}

\begin{lemma}\label{lem:ao-meager}
Let $X \subseteq \K^n$ be almost open.
Then, 
$X$ is nowhere dense;
iff $X$ has empty interior.
\end{lemma}
\begin{proof}
The ``if'' direction is clear..
Conversely, assume that $X$ has empty interior.
By definition, $X = U \Sdiff F$, for some open set $U$ and some definably
meager set $F$.
By Lemma~\ref{res:meager-n}, $F$ is nowhere dense.
Thus, since $X$ has empty interior, $U$ must be empty, and $X = F$.
\end{proof}

\begin{lemma}[$\Ds$-Uniformization]\label{lem:Fs-uniformization}
Let $A \subseteq \K^n \times \K^p$ be a $\Ds$ set.
Let $B \coloneqq \Pi^{n+p}_n(A)$.
Then, there exists a definable function $f: B \to \K^p$, such that
$\Disc(f)$ is nowhere dense and,
for all $b \in B$, $\pair{b, f(b)} \in A$.
\end{lemma}
\begin{proof}[Proof of Lemma~\ref{lem:Fs-uniformization}]
Since $A$ is a $\Ds$ set, there exists a definable family 
$\Pa{A(i): i \in N}$ of \dcompact sets, such that $A = \bigcup_{i \in N} A(i)$.
For every $i \in N$, define $B(i) \coloneqq \Pi^{n+p}_n(A(i))$,
and $C(i) \coloneqq B(i) \setminus \bigcup_{j < i} B(j)$.
Notice that $B = \bigcup_{i \in N} C(i)$.
For every $i \in N$, define $f_i: C_i \to \K^p$, $f_i(c) \coloneqq
\lexmin(A(i)_c)$, and define $f: B \to \K^p$ as 
$\Gamma(f) \coloneqq \bigcup_{i\in N} f_i$.
If $B$ has empty interior, then it is meager, and therefore nowhere dense, and
we are done.
Otherwise, let $U \coloneqq \bigcup_{i\in N} \interior(C_i)$.
Notice that $B \setminus U$ is definable and hence nowhere dense;
thus, it suffices to show that
$f \rest U$ is continuous outside a nowhere dense set.
By Lemma~\ref{lem:compact-uniformization}, each set $\Disc(f_i)$ is definably
meager, and thus nowhere dense.
Therefore, 
$F \coloneqq \bigcup_i \Pa{\Disc(f_i)}$ is also a definably meager
set, and thus nowhere dense; let $U' \coloneqq U \setminus \cl F$.

We claim that $f \rest {U'}$ is continuous.
It suffices to show that $f \rest{\interior(C_i)\setminus \cl F}$ is
continuous, for every $i \in N$.
However, $f \rest{\interior(C_i) \setminus \cl F} = 
f_i \rest{\interior(C_i) \setminus \cl F}$, and the latter is continuous by
definition of $F$.
\end{proof}

\begin{corollary}[$\Ds$-Choice]\label{cor:inverse}
Let $C \subseteq \K^n$ be a $\Ds$ set, and $f: C \to \K^p$ be a
definable continuous function.
Then, there exists 
a definable function $g: f(C) \to C$, 
 such that, for every $y \in f(C)$, $f(g(y)) = y$, and $\Disc(g)$ is nowhere dense.
\end{corollary}
\begin{proof}
Let
$
A \coloneqq \set{\pair{f(c), c}: c \in C}
$.
Notice that $A$ is a $\Ds$ set; the conclusion follows
by applying Lemma~\ref{lem:Fs-uniformization} to~$A$.
\end{proof}

In the above corollary, notice that, if $f(C)$ itself is nowhere dense, then it might happen that $\Disc(g) = f(C)$.

\section{Dimension and closure operator}\label{sec:dimension}
We will prove some good properties for the dimension function on $\Ds$
sets; in particular, we will prove Theorem~\ref{mainthm:2}.

\begin{lemma}\label{lem:dim}
Let $A \subseteq \K^n$ be a $\Ds$ set.
Then, $\dim(\cl A) = \dim(A)$.
\end{lemma}
\begin{proof}
Let $d \coloneqq \dim(A)$.
If $d = n$, the result is clear.
Thus, \wloG $d < n$.
Assume, for a contradiction, that $\dim(\cl A) = e > d$.
Let $L$ be a coordinate space of dimension~$e$,
such that $\Pi^n_L(\cl A)$ has nonempty interior inside~$L$.
Notice that $\Pi^n_L(\cl A) \subseteq \cll(\Pi^n_L(A))$,  that
$\Pi^n_L(A)$ is a $\Ds$ set, and that, by assumption,
$\Pi^n_L(A)$ has empty interior.
Thus, by Lemma~\ref{res:meager-n}, $\Pi^n_L(A)$ is nowhere dense, and therefore
$\Pi^n_L(\cl A)$ has empty interior, absurd.
\end{proof}
Notice that we cannot conclude that, if $A$ is a $\Ds$ set, 
then $\dim(\partial A) < \dim(A)$,
since the latter inequality fails for \dminimal structures.

\begin{lemma}\label{lem:dim-function-0}
Let $A \subset \K^n$ be a $\Ds$ set and $f: A \to \K^m$ be a
definable continuous function.
Assume that $\dim(A) = 0$.
Then, $\dim(f(A)) = 0$. 
\end{lemma}
\begin{proof}
Assume, for a contradiction, that $\dim(f(A)) > 0$.
\Wlog, $m = 1$.
Since $f(A)$ is a $\Ds$ set, this means that $f(A)$ is
nonmeager, and thus it contains an open interval~$I$.
By applying Corollary~\ref{cor:inverse},
we conclude that there exists an open interval $I' \subseteq I$ 
and a continuous definable function $g: I' \to \K^n$, such that,
for every $y \in I'$, $g(y)\in A$ and $f(g(y)) = y$.
Since $I'$ is definably connected and $g$ is continuous, $g(I')$ is also
definably connected.
Since $\dim(A) = 0$, the function $g$ must be constant, contradicting 
$f(g(y)) = y$.
\end{proof}

We write that a set is $\emptyset$-definable if it is definable without
parameters. 
We introduce a matroid, which is useful in treating the dimension for $\Ds$
sets.

\begin{definition}\label{def:Fcl}
Let $B \subset \K^n$ be any set (not definable, in general)
and $a \in \K$.
We say that $a \in \Fcl(B)$ (the ``F'' stands for ``$\mathcal F_\sigma$'') if
there exists $\bv \in B^n$ and $X \subset \K^{n+1}$, such that:
\begin{enumerate}[(a)]
\item\label{en:def-Fcl-Ds} $X$ is a $\Ds$ set and $X$ is \Zdefinable;
\item\label{en:def-Fcl-nd} For every $\y \in \K^n$, $X_y$ is nowhere dense;
\item $a \in X_\bv$.
\end{enumerate}
\end{definition}
Notice that, in the above definition, under Assumption~\enumref{en:def-Fcl-Ds}, 
Assumption~\enumref{en:def-Fcl-nd} is equivalent to: 
\begin{enumerate}
\item[(\ref{en:def-Fcl-nd}')]  For every $\y \in \K^n$, $X_y$ is definably meager; 
\end{enumerate}
and to:
\begin{enumerate}
\item[(\ref{en:def-Fcl-nd}'')] For every $\y \in \K^n$, $X_y$ has empty interior. 
\end{enumerate}

\begin{lemma}\label{lem:matroid}
$\Fcl$ is a finitary matroid: that is (for every $B$ and $C$ subsets of $\K$
and every $a,c \in \K$)
\begin{enumerate}
\item\label{en:mat-ext} $\Fcl$ is extensive: 
  $B \subseteq \Fcl(B)$;
\item\label{en:mat-incr} $\Fcl$ is increasing: 
  $B \subseteq C$ implies $\Fcl(B) \subseteq \Fcl(C)$;
\item\label{en:mat-idemp} $\Fcl$ is idempotent: 
  $\Fcl(\Fcl(B)) = \Fcl(B)$;
\item\label{en:mat-EP} $\Fcl$ satisfies the Exchange Property:
  $a \in \Fcl(B c) \setminus \Fcl(B) \Rightarrow c \in \Fcl(B a)$;
\item\label{en:mat-fin} $\Fcl$ is finitary:
  if $a \in \Fcl(B)$, then there exists $B' \subseteq B$ finite, such that $a
  \in \Fcl(B')$.
\setcounter{saveenum}{\value{enumi}}
\end{enumerate}
Moreover, we have:
\begin{enumerate}
\setcounter{enumi}{\value{saveenum}}
\item\label{en:mat-dim}
let $X \subset \K^n$ be a $\Ds$ set of dimension $d \leq n$,
which is \Zdefinable.
Let $\bv \in X$.
Then, there exists $\bv'$ a subtuple of $\bv$ of length $d$,
such that $\Fcl(\bv') = \Fcl(\bv)$.
\end{enumerate}
\end{lemma}
\begin{proof}
\enumref{en:mat-incr} and \enumref{en:mat-fin}  are clear.
\enumref{en:mat-ext} is also clear: 
take $X \coloneqq \set{\pair{x, y} \in \K^2: x = y}$ in the definition of $\Fcl$.

\enumref{en:mat-idemp} Let $a \in \Fcl(\Fcl(B))$.
Thus, there exist $\cv\in \K^n$ and $\bv \in B^m$,
such that $a \in \Fcl(\bv \cv)$ and, 
for every $i= 1, \dotsc, n$, $c_i \in \Fcl(\bv)$.
Therefore, by definition, there exist $X \subset \K^{m + n + 1}$ and $Y(i)
\subset \K^{m+1}$, $i = 1, \dotsc, n$, such that
\begin{enumerate}[(a)]
\item $X$ and $Y(i)$ are $\Ds$ and \Zdefinable set;
\item For every $\y \in \K^n$ and $\z \in \K^m$,
$X_{\y,\z}$ and $Y(i)_{\z}$ are nowhere dense;
\item $\pair{\bv, \cv, a} \in X$ and $\pair{\bv, c_i} \in Y(i)$.
\end{enumerate}
Define
\[
Z  \coloneqq X \cap \set{\pair{\y, \z, x} \in \K^m \times \K^n \times \K:
\smash{\bigwedge_{i = 1}^n} \pair{\y, z_i}\in Y(i)}, \qquad
W \coloneqq \Pi(Z),
\]
where 
$\Pi: \K^m \times \K^n \times \K \to \K^m \times \K$ is
the projection omitting the ``middle'' $n$ coordinates.

Notice that $\pair{\bv,\cv,a} \in Z$, and therefore $\pair{\bv, a} \in W$.
Moreover, it is clear that $W$ is a $\Ds$ and \Zdefinable set.
Thus, it suffices to show that, for every $\y \in \K^m$, $W_{\y}$ has empty interior.
Assume, for a contradiction, that $W_{\dv}$ has nonempty interior, 
for some $\dv \in \K^m$.
By $\Ds$-Uniformization, there exists an open interval 
$I \subseteq W_{\dv}$ and a continuous function $g: I \to \K^n$, such that,
for every $t \in I$, $\pair{g(t), t} \in Z_{\dv}$.
Since
\[
\Z_{\dv} \subseteq Y(1)_\dv \times \dots \times Y(n)_\dv 
\times \K,
\]
and each $Y(i)_\dv$ is nowhere dense, the function $g$ must be constant, 
say $g(t) = \bv'$
Thus, $\set{\bv'} \times \K \subseteq Z_{\dv}$, contradicting the assumption
that $Z \subseteq X$ and (b).

\enumref{en:mat-dim}
We define the full dimension of $X$ as the lexicographically ordered pair 
$\pair{d,k}$, where $d \coloneqq \dim(X)$ and $k$ is the number of $d$-dimensional 
coordinate spaces $L$, such that $\Pi^n_L(X)$ has nonempty interior.
The proof is by induction on the full dimension of~$X$.
If $d = n$, we take $\bv' = \bv$.
Otherwise, let $Y^{>0} \coloneqq \set{\y \in \K^d: \dim(X_{\y}) > 0}$, and
$X^{>0} \coloneqq X \cap \cll\Pa{\Pi^{-1}(Y^{>0})}$, where $\Pi \coloneqq \Pi^{n}_d$.
By Lemma~\ref{lem:dim-fiber-0}, $X^{>0}$ is a $\Ds$ set of full
dimension less than the full dimension of~$X$.
Thus, if $\bv \in X^{>0}$, then the conclusion follows by inductive
hypothesis.
Assume instead that $\bv \in X^{0} \coloneqq X \setminus X^{>0}$; 
define $\bv' \coloneqq \pair{b_1, \dotsc, b_d}$.
Consider for simplicity the case when $n = d+1$; notice that $X^{0}$ is a
$\Ds$ set with fibers of dimension $0$:
by definition, $b_n \in \Fcl(\bv')$.
In the general case we proceed similarly, by using
Lemma~\ref{lem:dim-function-0}, and conclude that $\bv \in \Fcl(\bv')$.
 The conclusion then follows from \enumref{en:mat-ext}, 
\enumref{en:mat-incr} 
and \enumref{en:mat-idemp}.

\enumref{en:mat-EP}
Let $\bv \subset B$ be of minimal length, such that $a \in \Fcl(\bv)$; say,
$\bv \in B^n$.
There exists 
$X \subseteq \K^{n+2}$, such that
$X$ is a \Zdefinable $\Ds$ set, 
for every $\pair{\y, z} \in \K^{n+1}$, $X_{\y,z}$ is nowhere dense, 
and $\pair{\bv, c, a} \in X$.

For every $\dv \in \K^n$,  define 
\[
Y_{\dv}^{>0} \coloneqq \set{x \in \K: \dim\Pa{\iset{z \in \K: \pair{\dv, z, x}
      \in X}}>0},
\]
and $Y^{>0} \subseteq \K^{n+1}$ be the set whose fibers are given by the 
$Y_{\dv}^{>0}$.
By Lemma~\ref{lem:dim-fiber-0}, each $Y^{>0}_{\dv}$ is nowhere dense; moreover,
the whole $Y^{>0}$ is a $\Ds$ set.

Assume that $a \in \cll(Y^{>0}_{\bv})$.
Notice that $Y^{>0}$ is a $\Ds$ set with empty interior, and
therefore it is nowhere dense.
Let $W \coloneqq \cll(Y^{>0})$; notice that $\pair{\bv, a} \in W$.
Since $W$ is a nowhere dense subset of $\K^{n+1}$, $\dim(W) \leq n$;
thus, by Lemma~\ref{lem:dim-fiber-0}, the set
$C \coloneqq \set{\dv \in \K^n: \dim(W_{\dv}) > 0}$ is nowhere dense; by
  minimality of $\bv$  and \enumref{en:mat-dim}, $\bv \notin \cl C$.
Let $W' \coloneqq W \setminus (\cl C \times \K)$.
Then, $\pair{\bv,a} \in W'$, and $W'_{\dv}$ is nowhere dense for every $\dv
\in \K^n$; thus, $a \in \Fcl(\bv)$, a contradiction.

Assume now that $a \notin \cll\Pa{Y^{>0}_{\bv}}$.
By Lemma~\ref{res:meager-n}, the set
$B \coloneqq\set{\dv  \in \K^n: \cll(Y_{\dv}^{>0}) \neq \cll(Y^{>0})_{\dv}}$ is
nowhere dense; 
thus, by \enumref{en:mat-dim} and minimality of $\bv$, 
we have $\bv \notin \cl B$,
and therefore $a \notin \cll(Y^{>0})_{\bv}$.
Let
\[
Z \coloneqq \set{\pair{\y, z, x} \in \K^n \times \K \times \K: \pair{\y, z, x}
\in X \et \pair{\y, x} \notin \cll(Y^{>0})}.
\]
Notice that $Z$ is a \Zdefinable $\Ds$ set, that 
$\pair{\bv, c, a} \in Z$, and that, for every $\dv \in \K^n$ and $a' \in \K$,
$\dim(\set{z\in \K: \pair{\dv, z,a'} \in Z}) \leq 0$.
Thus, $c \in \Fcl(\bv a)$.
\end{proof}

Since, by Lemma~\ref{lem:matroid}, $\Fcl$ is a matroid, it induces a 
\intro{rank} function, which we denote by $\rk$.

\begin{definition}
Let $\bv \in \K^n$.
Let $\K(\bv)$ be the expansion of $\K$ by constants denoting~$\bv$.
Let $\Fcl_{\bv}$ be the matroid $\Fcl$ defined in $\K(\bv)$
and $\rk_{\bv}$ be the rank associated to~$\Fcl_{\bv}$.
\end{definition}

\begin{remark}
We have $\Fcl(\bv) \subseteq \Fcl_{\bv}(\emptyset)$.
It is not true in general that $\Fcl(\bv) = \Fcl_{\bv}(\emptyset)$.
Therefore, $\rk_{\bv}(\av) \leq \rk(\av/\bv)$.
The relative (in $\K$) field algebraic closure of $\bv$ is contained in
$\Fcl(\bv)$, but the model-theoretic algebraic closure of $\bv$ might not be
contained in $\Fcl(\bv)$.
On the other hand, $\Fcl_{\bv}(\emptyset)$ contains the model-theoretic
algebraic closure of~$\bv$.
\end{remark}

\begin{lemma}\label{lem:rk-dim}
Assume that $\K$ is $\omega$-saturated.
\begin{enumerate}\item\label{en:rk-open} 
Let $U \subseteq \K^n$ be open, nonempty, and definable.
Then, there exists $\bv \in U$ such that $\rk(\bv) = n$.
\item\label{en:rk-dim-Ds}
Let $X \subseteq \K^n$ be a \Zdefinable $\Ds$ set.  Then,
$
\dim(X) = \max\Pa{\rk(\bv): \bv \in X}
$.
\end{enumerate}
\end{lemma}
\begin{proof}
\enumref{en:rk-open}
We proceed by induction on~$n$.
First, we consider the case $n = 1$.
Consider the following partial type (over the parameters of definition of~$U$):
\[
\Lambda(x) \coloneqq \set{x \in U \et x \notin Y: Y \subset \K \text{ nowhere
    dense and \Zdefinable}}.
\]
If $\Lambda(x)$ is consistent, then any realization $b$ of $\Lambda(x)$ will
satisfy $b\in U$ and $\rk(b) = 1$.
If, for a contradiction, $\Lambda(x)$ were inconsistent, there would exists
finitely many nowhere dense sets $X_1, \dotsc, X_k \subset \K$,
such that $U \subseteq X_1 \cup \dots \cup X_k$, which is absurd.

Assume now that we have proved \enumref{en:rk-open} for $n-1$; we want to prove
it for~$n$. Let $V \coloneqq \Pi^n_{n-1}(U)$. 
By inductive hypothesis, there exists $\cv \in V$ such that $\rk(\cv) = n-1$.
Add $\cv$ to the language, and consider the matroid $\Fcl_{\cv}$.
By applying the case $n=1$ to the open set $U_{\cv}$ and the matroid
$\Fcl_{\cv}$, we find $b_n \in U_{\cv}$, 
such that $b_n \notin \Fcl_{\cv}(\emptyset)$.
Therefore, $b_n \notin \Fcl(\cv)$.
Let $\bv \coloneqq \pair{\cv, b_n} \in \K^n$.
We have that $\bv \in U$ and $\rk(\bv) = n$.

\enumref{en:rk-dim-Ds}
Let $d \coloneqq \dim(X)$ and $e \coloneqq  \max\Pa{\rk(\bv): \bv \in X}$.
We prove that $d \geq e$ and $e \geq d$.

($d \geq e$).
Let $\bv \in X$ such that $\rk(\bv) = e$.
By Lemma~\ref{lem:matroid}(6), there exists $\bv'$ a subtuple of $\bv$ of
length~$d$, such that $\Fcl(\bv') = \Fcl(\bv)$.
Thus, $e = \rk(\bv) = \rk(\bv') \leq d$.

($e \geq d$).
Since $\dim(X) = d$, \wloG $\Pi^n_d(X)$ contains 
a nonempty definable open set~$U$.
By \enumref{en:rk-open}, there exists $\cv \in U$ such that $\rk(\cv) = d$.
Any $\bv \in X \cap (\set{\cv} \times \K^{n-d})$ will satisfy 
$\rk(\bv) \geq d$.
\end{proof}

\begin{lemma}\label{lem:dim-additive}
Let $A \subseteq \K^{m+n}$ be $\Ds$ and $B \coloneqq \Pi^{n+m}_n(A)$.
\begin{enumerate}
\item\label{en:add-geq} 
Assume that $\dim(B) \geq q$ and, for every $\bv \in B$, $\dim(A_{\bv})
\geq p$.
Then, $\dim(A) \geq p + q$.
\item\label{en:add-leq}
Assume that $\dim(B) \leq q$ and, for every $\bv \in B$, 
$\dim(A_{\bv}) \leq p$.
Then, $\dim(A) \leq p + q$.
\item\label{en:add-eq}
In particular, if $\dim(A_\bv) = p$ for every $\bv \in B$,
then $\dim(A) = \dim(B) + p$.
\end{enumerate}
\end{lemma}
\begin{proof}
\Wlog, $\K$ is $\omega$-saturated and $A$ is \Zdefinable.

\enumref{en:add-geq}
By Lemma~\ref{lem:rk-dim}, there exists
$\bv \in B$ such that $\rk(\bv) \geq q$.
By Lemma~\ref{lem:rk-dim} again, applied to the matroid $\Fcl_{\bv}$, 
there exists $\cv \in A_{\bv}$, such that $\rk_{\bv}(\cv) \geq p$.
Thus,
\[
\rk(\bv \cv) = \rk(\bv) +  \rk(\cv/ \bv) \geq \rk(\bv) +  \rk_{\bv}(\cv) \geq
p + q.
\]
The conclusion follows by applying Lemma~\ref{lem:rk-dim} a third time.

\enumref{en:add-leq}
First, we do the case when $p = 0$.
Let $\pair{\bv, \cv} \in A$; by Lemma~\ref{lem:rk-dim}, it suffices to show
that $\rk(\bv \cv) \leq p$.
Since $p = 0$, $\cv \in \Fcl(\bv)$, and therefore
$\rk(\bv \cv) = \rk(\bv)$.
By Lemma~\ref{lem:rk-dim}, $\rk(\bv) \leq p$, and we are done.

Next, we do the case when $p = m$.
It suffices to show that, for every $\pair{\bv, \cv} \in B \times \K^p$,
$\rk(\bv \cv) \leq p + q$.
Since $\bv \in B$, $\rk(\bv) \leq q$.
Since $\cv \in \K^p$, $\rk(\cv) \leq p$.
Thus,
\[
\rk(\bv \cv) = \rk(\bv) + \rk(\cv/\bv) \leq \rk(\bv) + \rk(\cv) 
\leq p + q.
\]

Now we do the general case by induction on 
$q$ and $p$.
Define $\Pi(m,p)$ be the set of orthogonal projection from $\K^m$ to some
$p$-dimensional coordinate space.
For every $\pi \in \Pi(m, p)$, let
$B(\pi) \coloneqq \set{\bv \in B: \dim(\pi(A_{\bv})) \geq p}$,
and $B(0) \coloneqq B \setminus \bigcup_{\pi \in \Pi(m,p)} \cll(B(\pi))$.
Correspondingly, $A(\pi) \coloneqq A \cap (B(\pi) \times \K^m)$ and
$A(0) \coloneqq A \cap(B(0) \times \K^m)$.
Let $\Pi'(m, p) \coloneqq \Pi(n,m) \cup \set 0$.
For every $\pi \in \Pi'(m, p)$, $A(\pi)$ is a $\Ds$ set.
Moreover, $A \subseteq \bigcup_{\pi \in \Pi'(m, p)} \cll(A(\pi))$;
thus, it suffices to prove that, for every $\pi \in \Pi'(m,p)$,
$\dim(\cll(A(\pi))) \leq p + q$.
Thus, by Lemma~\ref{lem:dim}, 
it suffices to show that, for every $\pi \in \Pi'(m,p)$,
\begin{equation}\label{eq:dim-additive}
\dim(A(\pi)) \leq p + q.
\end{equation}
Fix $\pi \in \Pi'(m,p)$.
If $\dim(B(\pi)) < q$, then  \eqref{eq:dim-additive} follows by induction 
on~$q$; therefore, we can assume $\dim(B(\pi)) = q$.

If $\pi = 0$, then,
by definition of dimension,
$\dim(A(0)_{\bv}) \leq p - 1$, for every $\bv \in B(0)$.
Therefore, by induction on $p$ we have $\dim(A(0)) \leq p + q - 1$.

Assume now that $\pi \in \Pi(m,p)$, and consider $A(\pi)$;
\wloG, $\pi = \Pi^m_p$.
Notice that the assumption in \enumref{en:add-leq} implies that,
for every $\bv \in B(\pi)$, $\dim(A(\pi)_{\bv}) = p$.
Let $D(\pi) \coloneqq \Pi^{n+m}_{n+p}(A(\pi))$.
Let $\bv \in B(\pi)$.
Notice that $D(\pi)_\bv = \pi\Pa{A(\pi)_\bv}$.
By the case $m = p$, $\dim(D(\pi)) \leq p + q$, and, by~\enumref{en:add-geq}, 
we have $\dim(D(\pi)) = p + q$.
Let
\[
D(\pi)'_{\bv} \coloneqq \set{\ev \in D(\pi)_{\bv}: \dim(A(\pi)_{\bv,\ev}) > 0},
\qquad A(\pi)' \coloneqq A(\pi) \cap (D(\pi)' \times \K^{m-p}).
\]
Let $D(\pi)'' \coloneqq D(\pi) \setminus \cll(D(\pi)')$ and 
$A(\pi)'' \coloneqq A(\pi) \cap (D(\pi)'' \times \K^{m-p})$.
By Lemma~\ref{lem:dim-fiber-0}, $A(\pi)'$ and $A(\pi)''$ 
are both $\Ds$ sets, and, 
by Lemma~\ref{lem:dim}, it suffices to show that 
$\dim(A(\pi)') \leq p + q$ and $\dim(A(\pi)'') \leq p + q$.

For every $\dv \in D(\pi)''$, $\dim(A_{\dv}) = 0$; thus,
by the case $p = 0$, applied to $A(\pi)''$ and $D(\pi)''$ 
instead of $A$ and~$B$,
$\dim(A(\pi)'') = \dim(D(\pi)'') \leq \dim(D(\pi)) \leq p + q$.

Let $R(A) = R(A; n, m, p)$ be the set of projections $\rho \in \Pi(m,p)$, such
that $\dim(\rho(A_{\bv})) \geq p$ for at least one $\bv \in B$.
We conclude by induction on the cardinality of~$R(A)$.
By Lemma~\ref{lem:dim-fiber-0}, $\dim(D(\pi)'_{\bv}) <  p$ for every $\bv \in B$.
Notice that $D(\pi)'_{\bv} = \pi\Pa{A(\pi)'_{\bv}}$,
and therefore $\pi \in R(A) \setminus R(A(\pi)')$;
by our inductive hypothesis, $\dim(A(\pi)') \leq p + q$, and we are done.
\end{proof}

\begin{corollary}\label{cor:dim-function}
Let $B \subseteq \K^n$ be a $\Ds$ set.
Let $f: B \to \K^m$ be definable and continuous.
Then, $\dim(f(B)) \leq \dim(B)$.
If, moreover, $f$ is finite-to-one, then $\dim(f(B)) = \dim(B)$.
\end{corollary}
\begin{proof}
Apply Lemma~\ref{lem:dim-additive} to
$A \coloneqq \Gamma(f)$.
\end{proof}

\section{The real case}\label{sec:real}
In this section we will prove most of the theorems in the introduction
which are specific to expansions of $\Rbar$
(since they mention Lebesgue measure and Hausdorff dimension).
Given a set $X \subseteq \K^n$, we define 
$X - X \coloneqq \set{x -y : x, y  \in  X}$.
Remember that $\Rcal$ is a \tame expansion of the real field.

\begin{lemma}\label{lem:tame-measure-1}
Let $C \subseteq \R$ be a nonempty $\Ds$ set.
Then, $\dim(C) = \dimH(C)$.
\end{lemma}
\begin{proof}
For every set, $\dimH(C) \geq \dim(C)$.
Thus, we have to prove that if $C$ is nowhere dense, then $\dimH(C) = 0$.
Since $\dim(C) = \dim(\cl C)$, \wloG $C$ is closed.
Conclude as in the proof of Theorem~\ref{mainthm:3}.
\end{proof}

\begin{lemma}\label{lem:tame-measure-n}
Let $D \subseteq \R^n$ be a $\Ds$ set.
Then, $\dim(D) < n$ iff $\cl D$ is null.
\end{lemma}

\begin{proof}
The ``if'' direction is clear.
So, assume that $D$ is nowhere dense; we have to show that $\cl D$ is null.
We proceed by induction on~$n$.
The case $n = 1$ is Lemma~\ref{lem:tame-measure-1}.
Assume that we have already proven the conclusion for $n-1$.
Assume that $D$ is nowhere dense. 
Since $\dim(\cl D) = \dim(D)$, \wloG $D$ is closed.
Let $F \coloneqq \set{\cv \in \K^{n-1}: D_{\cv} \text{ has positive Lebesgue measure}}$.
By the case $n = 1$,
$F = \set{\cv \in \K^{n-1}: \dim(D_\cv) = 1}$.
By  Lemma~\ref{lem:dim-fiber-0}, since $\dim(D) \leq n-1$,
we have that $F$ is nowhere dense in $\K^{n-1}$.
By inductive hypothesis, $F$ is null.
By Fubini's theorem, $D$ is null, and we are done.
\end{proof}

We now prove the ``moreover'' clause in Theorem~\ref{mainthm:1}.
We employ techniques similar to the one used in \cite{EM01}*{Lemma~1}.
We need some preliminary results from geometric measure theory.
\begin{fact}[\cite{mattila}*{Theorem~8.10}]\label{fact:H-additive}
Let $A \subseteq \R^n$ and $B \subseteq \R^m$ be $\Fs$ sets.
Then, $\dimH(A \times B) \geq \dimH(A) + \dimH(B)$.
\end{fact}
\begin{fact}[\cite{mattila}*{Corollary~9.8}]\label{fact:H-projection}
Let $A \subseteq \R^n$ be an $\Fs$ set.
Assume that $\dimH(A) \linebreak[0]> m$ (with $m < n$).
Then, there exists an $m$-dimensional linear subspace $L \subseteq \R^n$,
such that $\Pi^n_L(A)$ is not null.
\end{fact}


\begin{lemma}\label{lem:dim-H}
Let $X \subseteq \R^n$ be a nonempty $\Ds$ set.
Then, $\dim(X) = \dimH(X)$.
\end{lemma}
\begin{proof}
\begin{prclaim}
Let $Y \subseteq \R^m$ be a $\Ds$ set, such that $\dim(Y) = p$.
Then, $\dimH(Y) \leq p + 1$.
\end{prclaim}
We proceed by induction on $m$ and~$p$.
If $p \geq m - 1$, the result is clear.
Thus, \wloG $m > p + 1 \geq 1$.
Assume, for a contradiction, that $\dimH(Y) > p + 1$.
Thus, by Fact~\ref{fact:H-projection}, $Y' \coloneqq \Pi^n_L(Y)$ is not null,
for some $(p + 1)$-dimensional linear subspace~$L \subseteq \R^m$.
By Lemma~\ref{lem:tame-measure-n}, $\dim(Y') = p + 1$;
since $\dim(Y) \geq \dim(Y')$, we have a contradiction, and the claim is
proven.

The inequality $\dimH(X) \geq \dim(X)$ is true for any set $X \subseteq \R^n$;
thus, it suffices to prove the opposite inequality.
Assume, for a contradiction, that $\dimH(X) = \dim(X) + \eps$, with $\eps > 0$.
Let $m \in \N$ such that $m > 1/\eps$.
Let $Y \coloneqq X^m \subseteq \R^{nm}$.
By Fact~\ref{fact:H-additive}, $\dimH(Y) \geq n \dimH(X) \geq n p + 1 =
\dim(Y) + 1$, contradicting the claim.
\end{proof}


\begin{proof}[Proof of Theorem~\ref{mainthm:imin-dim-n}]
Let $D \coloneqq \cl C$.
By Lemma~\ref{lem:dim-H} (applied to~$D$),
$\dimH(D) = \dim(D) = \dim(C) \leq \dimH(C) \leq \dimH(D)$.
We also have to show that $C$ is definably almost open and Lebesgue measurable.
Notice that $E \coloneqq C \setminus \inter C$ is definable and has empty
interior.
Thus, by \cite{tame:3}*{\S3}, $E$ is nowhere dense.
Moreover, by Lemma~\ref{lem:tame-measure-n}
$E$ is null.
Thus, $C = \inter C \cup E$, where $\inter C$ is a definable open set and 
$E$ a definable set which is nowhere dense and null.
\end{proof}

\begin{proof}[Proof of Theorem~\ref{mainthm:sparse}]
Assume that $\pair{\Rcal, E}$ does not define~$\N$ (thus, $\pair{\Rcal, E}$ is
\tame).
Let $B \subset \R$ be definable in $\pair{\Rcal, E}^\#$ 
and assume that $B$ has no interior.
We must prove that $B$ is nowhere dense.

First, we do the case when $n = 1$.
By \cite{FKMS}*{1.11}, there exists $f: \R^m \to \R$ definable in~$\Rcal$,
such that $B \subseteq \cll(f(E^m))$.
Since we assumed that $E$ is closed and $0$-dimensional and $\pair{\Rcal, E}$
is \tame, we have that $\dim(f(E^n)) = 0$; thus, $f(E^n)$ is nowhere dense, and
we are done.

Assume now that $n \geq 1$.
Let $F$ be the union of the closures of the coordinate projections of~$E$:
notice that $F$ is closed, nowhere dense, and definable in $\pair{\Rcal, E}$,
and that $\pair{\Rcal, E}^\# = \pair{\Rcal, F}^\#$.
We conclude by applying the case $n =1$ to~$F$.
\end{proof}

\begin{question}[C. Miller]
Let $\Rc$ be an expansion of $\Rbar$.
Are the following equivalent?
\begin{enumerate}
\item\label{en:ques-imin} 
$\mathcal R$ is \iminimal;
\item\label{en:ques-Minkowski} 
every definable subset of $\R$ either has interior or has Minkowski 
upper dimension~$0$.
\end{enumerate}
\end{question}
In the above question, it is clear that \enumref{en:ques-Minkowski} 
implies~\enumref{en:ques-imin}; however, we don't
even know if \dminimal structures satisfy~\enumref{en:ques-Minkowski}.

\section{Continuous and differentiable functions}
Notice that we have little control on general definable functions:
there are examples of \tame expansions of $\R$ which define a function
$f: \R \to \R$ whose graph is dense in $\R^2$,  and such that $f$ is
discontinuous at every point: see \cite{vdd-dense}*{p.~62}.
However, continuous functions are much better behaved.

\begin{lemma}\label{lem:small-image}
Let $C \subseteq \K^n$ be definable and nowhere dense.
Let $f: \K^n \to \K^m$ be definable and continuous.
If $m \geq n$, then $f(C)$ is nowhere dense.
\end{lemma}
\begin{proof}
Let $B \coloneqq \cl C$.
By definition, $\dim(B) < m$.
By Corollary~\ref{cor:dim-function}, $\dim(f(B)) < m$, and therefore
$f(B)$ is nowhere dense.
\end{proof}

Thus, a Peano curve is not definable.
\begin{corollary}\label{cor:Peano-n}
Let $f: \K^n \to \K^m$ be definable and continuous.
If $m  > n$, then $f(\K^n)$ is nowhere dense.
\end{corollary}

\subsection{Differentiable functions}

Definable continuous functions are differentiable almost everywhere.
\begin{lemma}\label{lem:C-1}
Let $f: \K \to \K$ be definable and continuous.
Fix $p \in \N$.
Then, there exists $V \subseteq \K$ which is open, definable, and dense,
such that, 
\begin{enumerate}
\item\label{en:monotone} 
for every $I$ definably connected component of $V$, $f \rest I$ is either
constant, or strictly monotone;
\item\label{en:Cp}
$f$ is $\Cp$ on $V$.
\end{enumerate}
\end{lemma}
\begin{proof}[Proof of Lemma~\ref{lem:C-1}]
The proof of \enumref{en:monotone} follows easily as in
\cite{miller}*{Thm.~3.3}, by using $\Ds$-Choice.

Now we prove \enumref{en:Cp}.
It suffices to prove the case when $p=1$.
Let $F$ be the closure of
\[
\set{\pair{t,x,z}: t \neq 0, z = \frac{f(x+t) - f(x)}t}
\]
inside $\K \times \K \times \K_\infty$, where $\K_\infty = \K \cup \set{\pm
  \infty}$.
Let $G \coloneqq F_0 = \set{\pair{x,z}: \pair{0,x,z} \in F}$.
Define the functions $g_l, g_r: \K \to \K_\infty$,
$g_l(x) \coloneqq \min(G_x)$, $g_r(x) \coloneqq \max (G_x)$.
By Lemma~\ref{lem:compact-uniformization}, $g_r$ and $g_r$ are continuous
outside a nowhere dense definable set~$D$.
Let $I$ be a definably connected component of $\K \setminus \cl D$.
It suffices to show that, after maybe ignoring a nowhere dense definable set,
 $g_l = g_r$ on $I$ and that they are finite.
Let $X \coloneqq \set{x \in I: g_r(x) = + \infty}$.
Notice that $X$ must have empty interior, and thus it is nowhere dense.
Similarly, the set $\set{x \in I: g_r(x) = - \infty}$ is nowhere dense.
Thus, after shrinking $I$, we can assume that $g_r$ and $g_l$ assume only
finite values on~$I$.
Let $Y \coloneqq \set{x \in I: g_l(x) < g_r(x)}$.
It suffices to show that $Y$ has empty interior.
Assume not: by continuity, and since $g_l \leq g_r$,
there exists an open  subinterval $J \subseteq I$ and a constant $c \in \K$,
such that, for every $x \in J$, $g_l(x) < c < g_r(x)$.
Consider now the function $h(x) \coloneqq f(x) - c x$, $h: J \to \K$.
By \enumref{en:monotone}, after maybe shrinking $J$ to a smaller subinterval,
$h$ is either constant or strictly monotone.
However, $h$ constant  contradicts $g_r > c$,
$h$ strictly increasing  contradicts $g_l < c$,
and $h$ strictly decreasing  contradicts $g_r > c$.
\end{proof}

\begin{exercise}
Let $f: \K \to \K$ be definable and monotone (but not necessarily continuous).
Then, there exists a closed definable nowhere dense set $F$ such that $f$ is
$\Cone$ outside~$F$.
\end{exercise}

\begin{proof}[Proof of Theorem~\ref{mainthm:Cp}]
\enumref{en:Cp-n}
Same proof as in Lemma~\ref{lem:C-1}\enumref{en:Cp}.

\enumref{en:Sard}
Notice that $\Sigma_f$ is a $\Ds$ set.
The result follows easily from $\Ds$-choice: see \cite{tame:3}*{\S4} 
for details.
\end{proof}

\section{Further conjectures}

\begin{conjecture}\label{conj:dim}
Let $\K$ be \tame.
There is a dimension function $d$ in the sense of
in the sense of \cite{dries89}.
That is, $d$ is a function assigning to every definable set a number in
$\N \cup \set {-\infty}$, satisfying the following axioms:
for every definable sets $A$ and $B \subseteq \K^n$ and $C \subseteq \K^{n+1}$,
\begin{enumerate}[\rm({Dim} 1)]
\item $d(A) = - \infty$ iff $A = \emptyset$,
$d(\set{a}) = 0$ for each $a \in \K$,
$d(\K) = 1$.
\item $d(A \cup B) = \max(d(A), d(B))$;
\item $d(A^\sigma) = d(A)$ for each permutation $\sigma$ of
$\set{1, \dotsc, n}$.
\item Define $C(i) \coloneqq \set{\x \in \K^n: d(C_\x) = i}$, $i = 0, 1$.
Then, each $C(i)$ is definable, and
$d\Pa{C \cap (C(i) \times \K)} = d(C(i)) + i$, $i = 0,1$.
\end{enumerate}
Moreover, $d$ coincides with $\dim$ on $\Ds$ sets.
\end{conjecture}

See \cite{tame:3} for several cases when we know that
the above conjecture holds.
Notice that if a function $d$ as in Conjecture~\ref{conj:dim} exists, it is
unique, and satisfies:
\begin{sentence}
If $X \subset \K$ is definable, then $d(X) = 1$ iff $F(X^4) = \K$, where
\[
F(x_1, x_2, y_1, y_2) \coloneqq
\begin{cases}
\frac{x_1 - x_2}{y_1 - y_2} & \text{if } y_1 \neq y_2;\\
0 & \text{otherwise}.
\end{cases}
\]
\end{sentence}
Conjecture \ref{conj:dim} implies the following conjecture.
\begin{conjecture}
Let $C$ be a $\Ds$ set, and $f: C \to \K^n$ be definable (but
not necessarily continuous).
Then, $\dim(f(C)) \leq \dim(C)$.
In particular, there is no surjective definable function between $\K^n$ and
$\K^{n+1}$.
\end{conjecture}
See Corollary~\ref{cor:dim-function} for a partial result in the
direction of the above conjecture.

We introduced before \iminimal expansions of $\Rbar$: the definition extends
to DC structures in the obvious way.
For more on \iminimal structures outside~$\R$, see \cite{tame:3}.

\begin{conjecture}
The open core of $\K$ is \iminimal.
\end{conjecture}
See lemmas~\ref{lem:ao-meager} and~\ref{lem:dim} for partial results in the
direction of the above conjecture.

The above conjectures are open even for the case when $\K$ expands~$\Rbar$.


\section{Pathologies}\label{sec:pathology}
In this section I will give a brief exposition (far from a complete one) of
``pathological'' phenomena in \tame structures, which may
contradict our choice of the nomenclature ``\tame{}''.

\begin{asparaenum}
\item\label{en:RN} 
The structure $\pair{\Rbar, \N}$ is wild, the class of its definable sets
coincides with the class of projective sets \cite{kechris}*{Exercise~37.6}:
(thus, a descriptive set theorist might say that $\pair{\Rbar,\N}$ is not
\emph{that} wild, after all).
\item
Let $F$ be a proper real closed subfield of~$\R$, and $\pair{\Rbar,F}$ be the
expansion of $\Rbar$ with a unary predicate for $F$.  Then, 
$\pair{\Rbar, F}$ has o-minimal open core, and thus it is \tame
(\cites{vdd-dense, DMS}).  Notice that $F$ is definable subset of $\R$ which is
both dense and codense.  
\begin{enumerate}
\item 
If we take $F$ is countable, then $F$ is an $\mathcal F_\sigma$ set which is
not $\Ds$ (in $\pair{\Rbar,F}$).  
\item
If we take $F$ not Lebesgue measurable (resp., not projective),
we have an example of a \tame structure ($\pair{\Rbar, F}$) which defines a set
which is not Lebesgue measurable (resp., not projective).%
\footnote{Since there are only $2^{\aleph_0}$
projective sets, and $2^{2^{\aleph_0}}$ real closed subfield of~$\R$, there
are many real closed non projective subfield of~$\R$. 
On the other hand, I don't know how to prove the existence of a
non-measurable real closed subfield of~$\R$.}
\end{enumerate}

\item
\cite{FKMS} give an example of an \iminimal (and hence \tame) expansion of
$\Rbar$ that defines a Borel isomorph of $\pair{\Rbar, \N}$.
\setcounter{saveenum}{\value{enumi}}
\end{asparaenum}

The next examples are about \tame structures outside the reals.
\begin{asparaenum}
\setcounter{enumi}{\value{saveenum}}
\item 
There exists an ultraproduct $\K$ of \ominimal structures,
such that $\K$ has the Independence Property (notice that $\K$ will be locally
\ominimal, and hence \tame).
\item
In \cite{HP07} Hrushovski and Peterzil produce an \ominimal structure $\K$
(outside~$\R$) and a first order sentence which is true in $\K$ but fails in any possible interpretation over the field of real numbers. 
\end{asparaenum}


\bibliographystyle{alpha}	
\bibliography{tame}		

\end{document}